\documentclass{amsart}

\usepackage{graphicx}  
\usepackage{subfigure}
\usepackage{amsmath}
\usepackage{amsthm}
\usepackage{amsfonts}

\usepackage{hyperref}
\usepackage{verbatim}
\usepackage{hhline}
\usepackage{float}

\newtheorem{theorem}{Theorem}[section]

\theoremstyle{definition}

\newtheorem{algorithm}[theorem]{Algorithm}

\theoremstyle{remark}

\numberwithin{equation}{section}

\begin{document}

\title{A Convex Surface with Fractal Curvature}

\author{Iancu Dima}
\address{Department of Mathematics, Ithaca College, 953 Danby Road, Ithaca, NY 14850, USA}
\email{idima1@ithaca.edu}
\thanks{ID was supported by the National Science Foundation through the Research Experience for Undergraduates (REU) Program, Grant DMS-1156350.}

\author{Rachel Popp}
\address{The Graduate Center, CUNY, 180 Queens Gate, 365 Fifth Avenue, New York, NY 10016, USA}
\email{rachel.popp@gmail.com}
\thanks{RP was supported in part by the National Science Foundation, Grant DMS-1162045}

\author{Robert S. Strichartz}
\address{Department of Mathematics, Cornell University, Malott Hall, Ithaca, NY 14853, USA}
\email{str@math.cornell.edu}
\thanks{RSS was supported in part by the National Science Foundation, Grant DMS-1162045.}

\author{Samuel C. Wiese}
\address{Department of Mathematics, Universit\"at Leipzig, Augustusplatz 10, 04109 Leipzig, Germany}
\email{sw31hiqa@studserv.uni-leipzig.de}
\thanks{SCW was supported by the Foundation of German Business (SDW)}

\subjclass[2000]{35P05}

\date{July 8, 2019.}

\keywords{Laplacian, eigenvalues, eigenfunctions, polyhedra, curvature}

\begin{abstract}
We construct a surface that is obtained from the octahedron by pushing out 4 of the faces so that the curvature is supported in a copy of the Sierpinski gasket in each of them, and is essentially the self similar measure on SG. We then compute the bottom of the spectrum of the associated Laplacian using the finite element method on polyhedral approximations of our surface, and speculate on the behavior of the entire spectrum.
\end{abstract}

\maketitle


\section{Introduction.}
A convex surface in 3-space may be regarded as an Alexandrov space, and so it has a well-defined curvature (as a measure) and a well-defined Laplacian. In this paper we construct an example of such a surface whose curvature is a fractal measure related to the Sierpinski Gasket (SG). For an introduction see the first chapter of \cite{S}. Our construction produces the surface $S$ as a limit of convex polyhedra $P_n$. The curvature of each $P_n$ is a discrete measure supported on its vertices, and these discrete measures will converge to the fractal measure on $S$. Each $P_n$ has a Laplacian, and the limit is the Laplacian on $S$. We compute the spectra of each of the Laplacians on $P_n$ and thereby obtain approximations to the spectrum of the Laplacian on $S$.

The first approximating polyhedron (with 6 vertices) $P_0$ is just a regular octahedron. Half of the faces will remain flat in the construction, but with folding. Of course folding does not introduce curvature. The other four alternating faces will have a SG inscribed in them, and then the “upside down triangles” in the SG picture will be pushed out, step-by-step, as $n$ increases. When $n=1$ just a single central triangle will be pushed out of each face. The result will be a polyhedron $P_1$ with 18 vertices, the original 6 together with 3 new vertices along the sides of the 4 faces (The flat faces will fold to accomodate the eruptions in the other curved faces). If we choose the angles correctly, the original vertices will have their curvature reduced from $\frac{2\pi}{3}$ to $\frac{\pi}{3}$, while the new vertices will have curvature $\frac{\pi}{6}$. Note that by the Gauss-Bonet Theorem the total curvature measure on $P_1$ will then be the sum of the discrete measures on the 4 faces with weight $\frac{\pi}{6}$ on each of the 6 vertices on the face (the original vertices lie in exactly 2 of these faces). Thus, we have the first discrete approximation to a standard self-similar measure with total mass $\pi$ on the SG pictured on these faces.

The same idea of pushing out successively smaller triangles to create $P_n$ from $P_{n-1}$ produces polyhedra whose curvature gives better and better approximations to the self-similar SG measures. The details are given in Section 2. In Section 3 we describe the finite element method we use to approximate the spectra of each $P_n$. In Section 4 we discuss the numerical results of these computations.

\section{Geometry}

\subsection{Construction of angles}
Each of the four curved faces in $P_n$ is subdivided into $3^m$ triangles of the SG construction labeled by words $w=(w_1,w_2,\dots,w_m)$ of length $m$, together with “upside down” equilateral triangles that will be flat. So let $T_w$ denote the triangle associated with the word $w$, namely $T_w = F_{w_1}\circ F_{w_2} \circ\dots\circ F_{w_m}T$ where $T$ is the full face, and let $\alpha_j(T_w)$ denote the angles at the vertices for $j=0,1,2$. We will choose
\begin{equation}
\alpha_j(T_w)=\frac{\alpha}{3}+(b(w,j)-1)\frac{\pi}{3^{m+1}}
\end{equation}
where $b(w,j)$ are integers to be described below.
\begin{theorem}
Suppose we can find values of $b(w,j)$ that are symmetric under permutations of $(0,1,2)$ and satisfy the following conditions:
\begin{align}
b(w,0)+b(w,1)+b(w,2)&=3 \text{ for all $w$},\\
b(w'_i,j)+b(w'_j,i)&=0 \text{ for all $w'$ of length $m-1$},\\
b(0,0)&=3^m \text{where $0$ denotes the word $(0,0,\dots,0)$.}
\end{align}
Then the curvature of $P_m$ is $\frac{2\pi}{3^{m+1}}$ at every vertex.
\end{theorem}
\begin{proof}
Condition (2.2) shows that the angle sum for each triangle $T_w$ is $\pi$. Now consider a vertex of the original octahedron. Then it has two angles of $\frac{\pi}{3}$ each from the flat faces and two angles of $\frac{\pi}{3}+(3^m-1)\frac{\pi}{3^{m+1}}$ at the curved faces by (2.4) and its symmetric versions. Thus the total angle sum is $\frac{4\pi}{3}+\frac{2\pi}{3}-\frac{2\pi}{3^{m+1}}=2\pi-\frac{2\pi}{3^{m+1}}$ and so the curvature there is $\frac{3\pi}{3^{m+1}}$.
Next consider a vertex interior to one of the faces. Then it has one angle $\frac{\pi}{3}$ from the equilateral triangle that touches at that point, and one angle $\pi$ from the edge of a triangle, and then by (2.3) the remaining two angles add to $\frac{2\pi}{3}-\frac{\pi}{3^{m+1}}$, so again the curvature is $\frac{2\pi}{3^{m+1}}$. Here we use the observation that the last two angles have the form $\alpha_j(T_{w'i})$ and $\alpha(T_{w'j})$.
\end{proof}
Our goal is to construct the constant $b(w,j)$ satisfying the conditions of the theorem. By the symmetry it suffices to do this for words $w$ with the initial $w_1=0$, so $w=0w'$. We show a solution for $m=2$ and $m=3$, filling in the $b$-value in each
angle in base-3 notation (Fig. \ref{fig:bvals}).
\begin{figure}
\includegraphics[width=\linewidth]{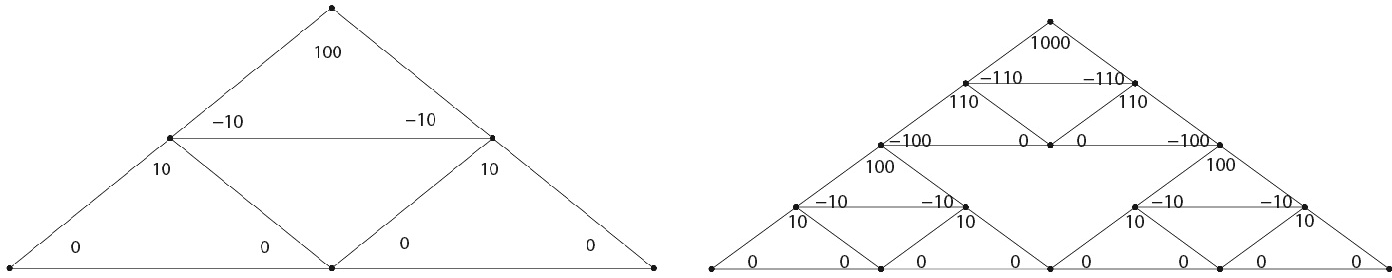}
\caption{The values of $b$ for level $m=2$ (left) and $m=3$ (right)}
\label{fig:bvals}
\end{figure}
The conditions are easily verified by inspection, and since $b(01,1)=0$ and $b(02,2)=0$, condition (2.3) holds at the bottom vertices for $m=2$ (similarly since $b(011,1)=b(022,2)=0$ for $m=3$) when we join up with the symmetric lower thirds of the face. We observe that the $b$-values on each of the two lower thirds for $m=3$ are identical to the whole set of $m=2$ values. This will be true in general when we describe an inductive construction from level $m$ to level $m+1$, so we only need to describe the $b$-values on the top third. We also note that all $b$-values are divisible by $3$, and the only base-3 digits are $0$ and $1$.

If we we examine the values of $b(w,0)$ along the right boundary edge, that is each $w_j=0$ or $1$, we see a decrease from top to bottom. For $m=3$ we have $b(000,0)=1000$, $b(001,0)=110$, $b(010,0)=100$, and $b(011,0)=10$. If we interpret $w$ as a base-2 number then these increase $0$, $01$, $10$, $11$ and if we consider $2^{m-1}-w$ these decrease $100$, $11$, $10$, $1$. Note that if we add a final $0$ and interpret as a base-3 number we get the correct $b$-values. In our first algorithm we extend this to the case of general $m$.
\begin{algorithm}
For any integer base-2 written as
\begin{align}
n&=\sum_{k=1}^N 2^{J_k} \text{for $0\leq J_1<J_2<\dots<J_N$ define}\\
b_n&=3\sum_{k=1}^N 3^{J_k}\text{.}
\end{align}
This is rereading the base-$2$ representation as base-$3$ and adding the digit $0$ at the end. Then set
\begin{equation}
b(w,0)=b_n
\end{equation}
if $w=0w'$ is a word of length $m$ consisting of digits $0$ or $1$ only, and
\begin{equation}
n=2^{m-1}-w\text{.}
\end{equation}
\end{algorithm}
Next we describe the inductive algorithm that computes $b$-values on level $m+1$ given the values on level $m$.
\begin{algorithm}
By symmetry we may assume $w_1=0$, so $w=0\tilde{w}$ with $|\tilde{w}|=m$. We consider three cases:
\begin{enumerate}
\item $\tilde{w}$ consists of digits $0$ and $1$ only. Then we apply Algorithm 2.2 to set $b(0\tilde{w},0)=b_n$. We then set
\begin{equation}
b(0\tilde{w},1)=-b_{n-1}
\end{equation}
so that (2.3) holds at the vertex, and
\begin{equation}
b(0\tilde{w},2)=3-b_n-b_{n-1}
\end{equation}
so that (2.2) holds on the triangle. Similarly, if $\tilde{w}$ consists of digits $0$ and $2$ only, we define $b(0\tilde{w},j)$ by symmetry.
\item If $\tilde{w}=1w'$, then set
\begin{equation}
b(01w',j)=b(0w',j)
\end{equation}
and similarly $b(02w',j)=b(0w',j)$. Note that this gives that bottom two thirds at level $m+1$ equal to the whole of level $m$.
\item If $\tilde{w}=0w'$ and both digits $1$ and $2$ occur in $w'$, then set
\begin{equation}
b(00w',j)=b(0w',j)
\end{equation}
\end{enumerate}
\end{algorithm}
\begin{theorem}
The $b$-values constructed by Algorithms 2.2 and 2.3 satisfy the conditions of Theorem 2.1.
\end{theorem}
\begin{proof}
Since we know this holds for $m=2,3$ we may proceed by induction. Every triangle at level $m+1$ in cases $(2)$ and $(3)$ has identical $b$-values to a triangle at level $m$, so (2.2) holds. For triangles in case $(1)$ we obtain (2.2) by adding (2.8), (2.9), and (2.10). Also (2.4) holds by Algorithm 2.2.

It remains to verify (2.3). For vertices in the interior of the lower thirds this follows by the induction hypothesis. For the vertex where they intersect this is the condition $b(0122\dots2,2)+b(0211\dots1,1)=0$, but it is easy to see by induction that both values are zero. For vertices along the left and right edges (2.3) is an immediate consequence of (2.8) and (2.9).

If $T$ is a triangle along the right edge we will verify (2.3) at the bottom left vertex. Here we consider two separate cases. If $\tilde{w}$ ends in the digit $1$ then $b_n+b_{n-1}=3$, and so (2.10) means $b(0\tilde{w},2)=0$. If we write $\tilde{w}=w'1$ then $b(0w'2,1)=0$ by local symmetry (easily established by induction) so (2.3) holds. On the other hand, if $\tilde{w}$ ends in the digit $1$, then (2.6) implies $b(0\tilde{w},2)=-b_{n-1}$ while the $b$-value on the adjacent is $b_{n-1}$ by local symmetry. Finally, for vertices further in the interior, condition (2.3) follows by induction.
\end{proof}
\subsection{Computing the sidelengths}
After we compute the $b$-values inductively, we get all the angles from (2.1) and set up a system of equations using the law of sines, symmetry, and norming the distance between the points associated to the angles $\alpha_0(T_0)$ and $\alpha_1(T_1)$ to be $1$, where $0=(0,\dots,0)$, $1=(1,\dots,1)$. The system will be uniquely solved for all the sidelengths. We also compute the area of every subtriangle that will be needed for the FEM.
\subsection{Assembling the polyhedra}
We use the sidelengths and angles to display the net of a level $m$ polyhedron, that will be triangulated using \cite{JS}. Figure \ref{fig:sgface} shows the low-level triangulation of one SG face.
\vspace{0.2cm}
\begin{figure}[H]
\minipage{0.32\textwidth}
\centering
  \includegraphics[width=3.7cm]{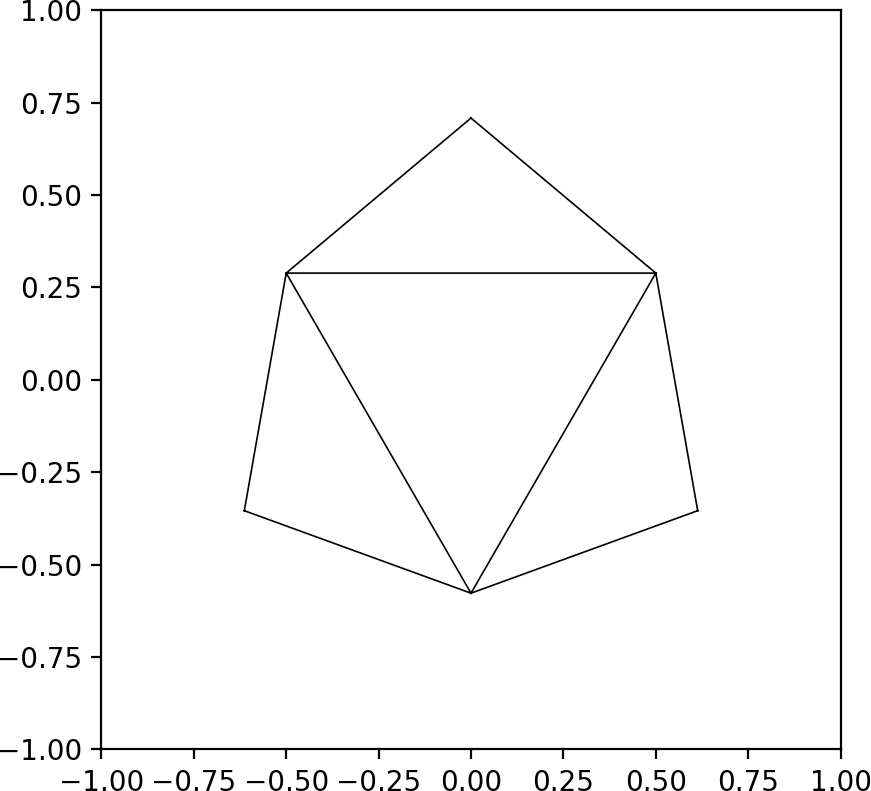}
\endminipage\hfill
\minipage{0.32\textwidth}
\centering
  \includegraphics[width=3.7cm]{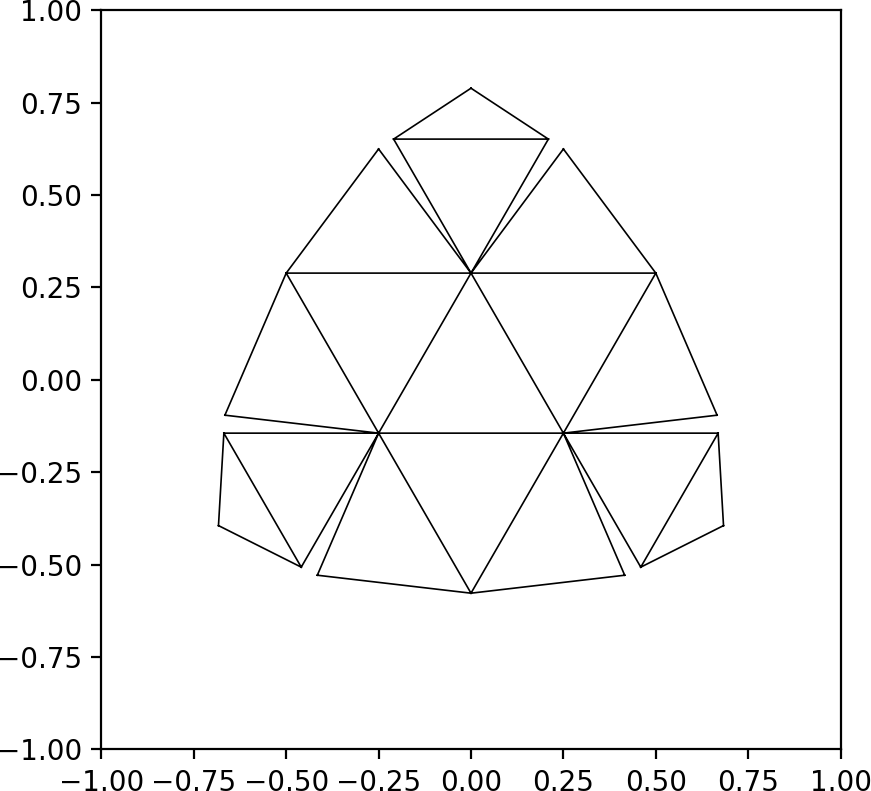}
\endminipage\hfill
\minipage{0.32\textwidth}
\centering
  \includegraphics[width=3.7cm]{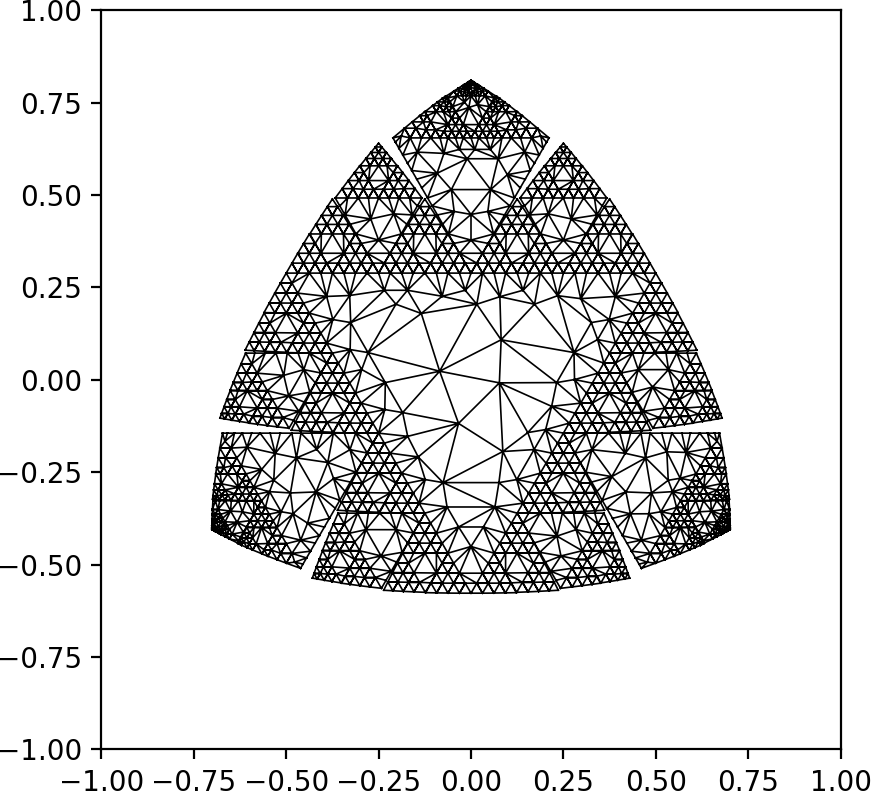}
\endminipage\hfill
\caption{Triangulation of one SG-face of level $m=1, 2, 6$}
\label{fig:sgface}
\end{figure}
%
%
\clearpage
To put the 4 SG faces and 4 flat faces together, we will need make identifications "inside" the SG-faces to stich up the small slits and also identify the sides of SG-faces and flat faces. Figure \ref{fig:all} shows the complete mesh with and without identifications, where two vertices are identified when they are connected by a blue line. This is done in a way that the 2*2 left faces are rotated $90^\circ$ to the right and put "on top" of the 2*2 faces on the right to form the polyhedron.
\begin{figure}[H]
\minipage{\textwidth}
\centering
\includegraphics[width=0.8\linewidth]{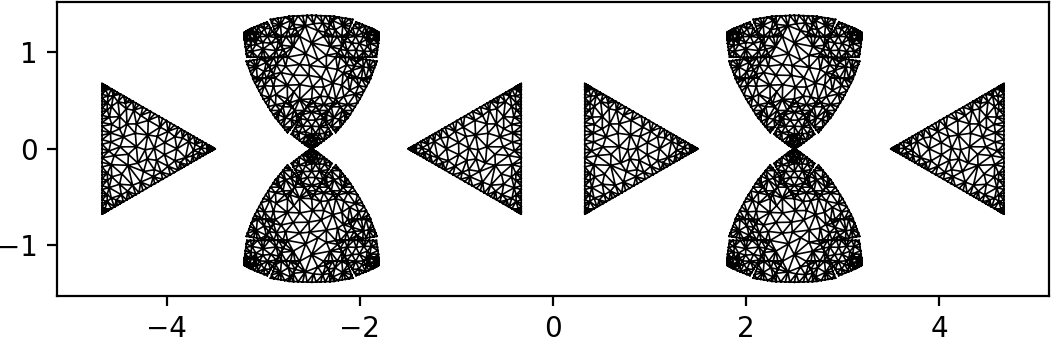}
\endminipage\hfill
\vspace{0.2cm}
\minipage{\textwidth}
\centering
\includegraphics[width=0.8\linewidth]{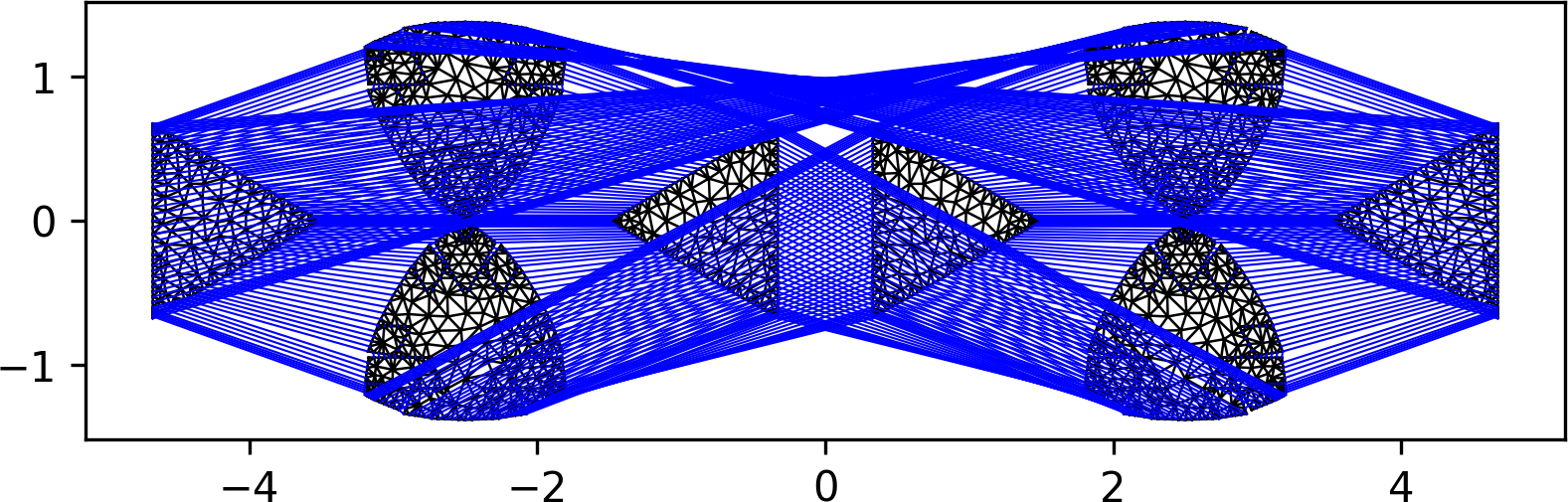}
\endminipage\hfill
\caption{The complete mesh with identifications $m=2$}
\label{fig:all}
\end{figure}
We numerically fold up the net using \cite{BI} to display the polyhedra in 3D. Figure \ref{fig:3d} shows level $m=1,2,3$, where for $m=3$ only the edges belonging to the flat faces are shown. Note that for $m\geq2$ folding occurs on the inner flat faces too, but we won't show them since they don't contribute to the curvature. Although there is zero curvature on the non-SG (flat) faces, the vertices added to the SG-faces cause the non-SG faces to fold. This creates trapazoids which must be triangulated for the FEM.

The solid has 3 symmetry axes through two opposing tops (by $\pi$) and 4 symmetry axes through the middles of two opposing faces (one flat, one curved) by $\frac{2}{3}\pi$ and $\frac{4}{3}\pi$. It's symmetry group is $S_4$ with order 24. We represent it as
\begin{equation*}
<a,b:a^2=1,b^3=1,(ab)^4=1>
\end{equation*}
and write the generators as matrices in $O(3)$ in a coordinate systems where the center of the solid is at the origin and the coordinate axes run through the vertices where 4 faces meet. Choose
\begin{equation*}
a=\begin{pmatrix} -1 & 0 & 0\\ 0 & -1 & 0 \\ 0 & 0& 1 \end{pmatrix}\text{, }b=\begin{pmatrix} 0 & 0 & 1\\ 1 & 0 & 0 \\ 0 & 1& 0 \end{pmatrix}
\end{equation*}
then $a$ corresponds to a reflection through the z-axis and $b$ fixes an opposing pair of vertices and cyclically permutes the others. The group $S_4$ has 2 irreducible representations of degree 1, 1 of degree 2, and 2 of degree 3.
\begin{figure}[!htb]
\minipage{0.32\textwidth}
  \includegraphics[width=3.7cm]{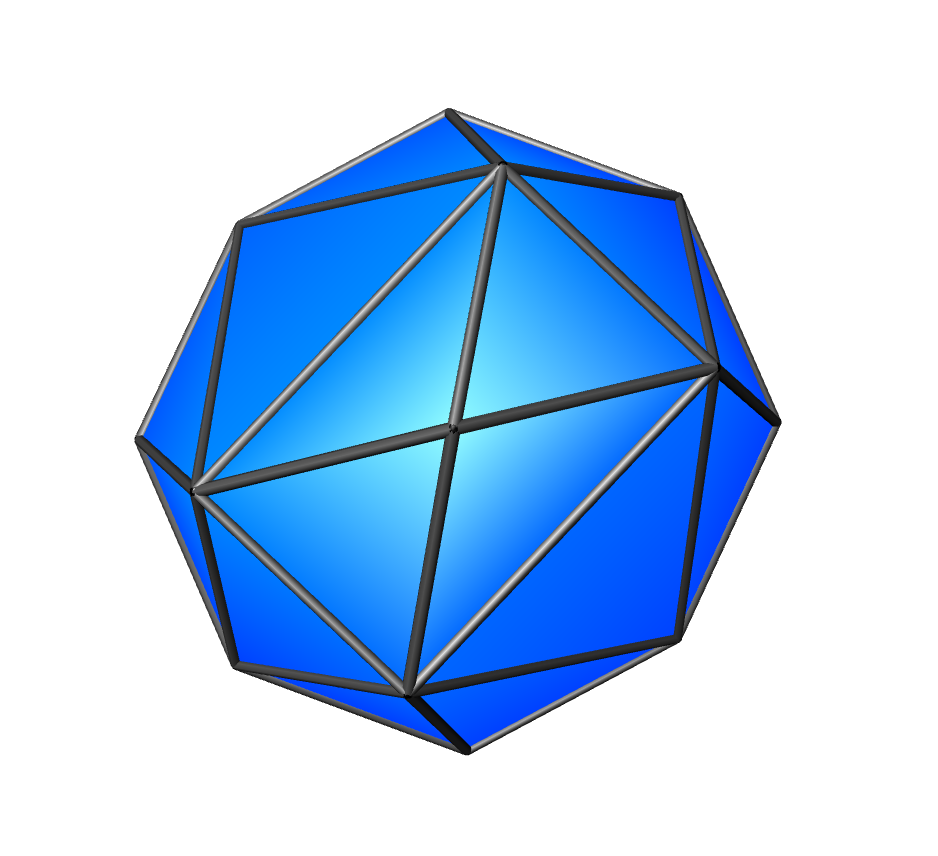}
\endminipage\hfill
\minipage{0.32\textwidth}
  \includegraphics[width=3.7cm]{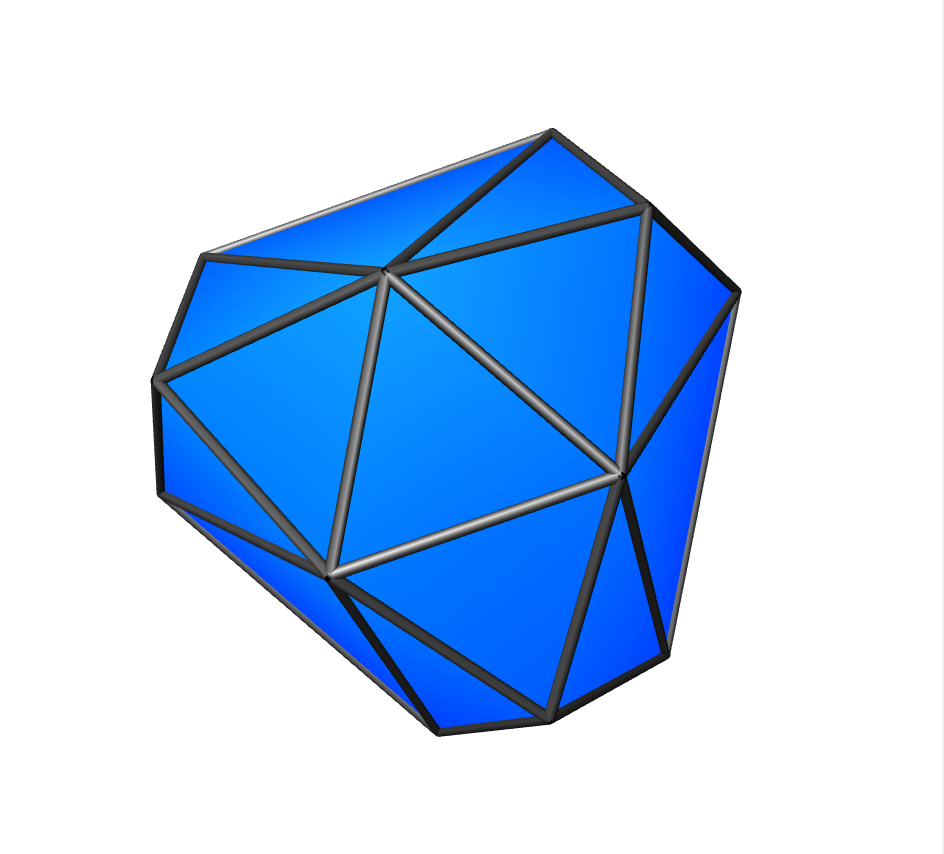}
\endminipage\hfill
\minipage{0.32\textwidth}
  \includegraphics[width=3.7cm]{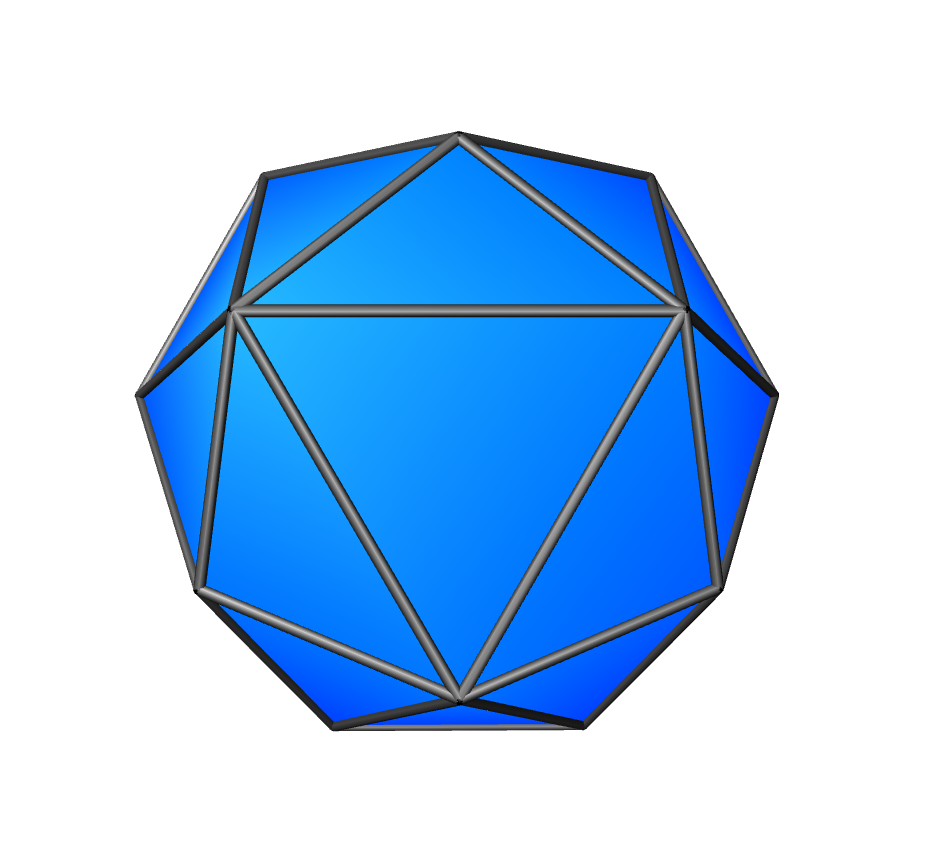}
\endminipage\hfill
\minipage{0.32\textwidth}
  \includegraphics[width=3.7cm]{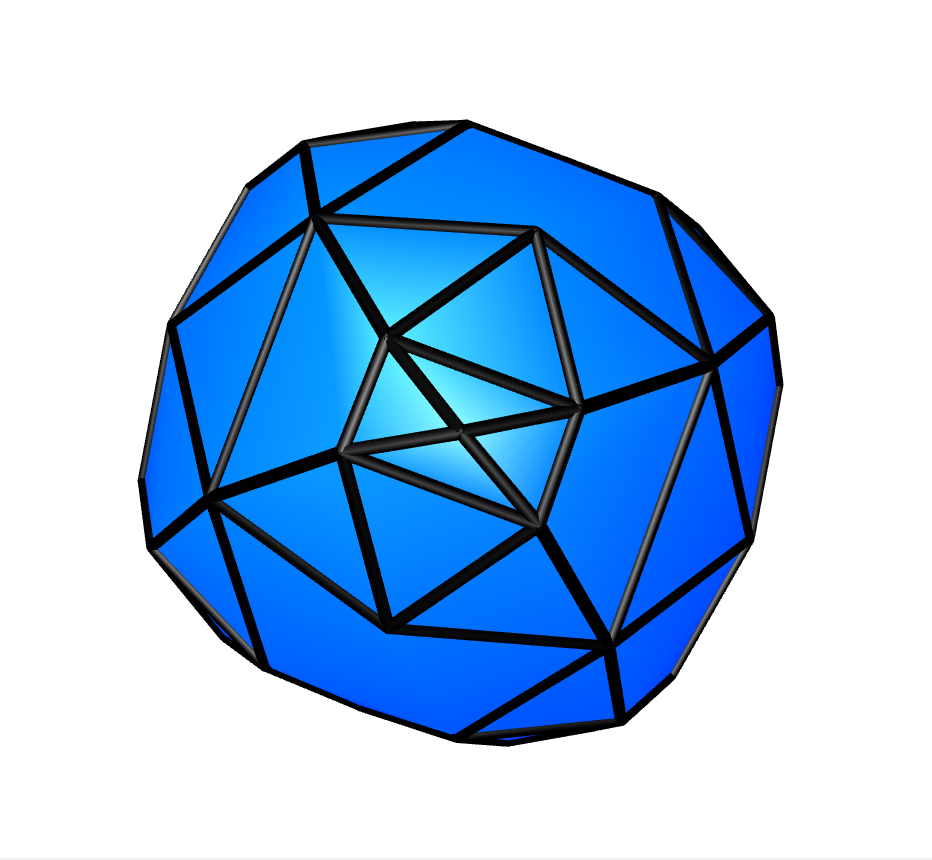}
\endminipage\hfill
\minipage{0.32\textwidth}
  \includegraphics[width=3.7cm]{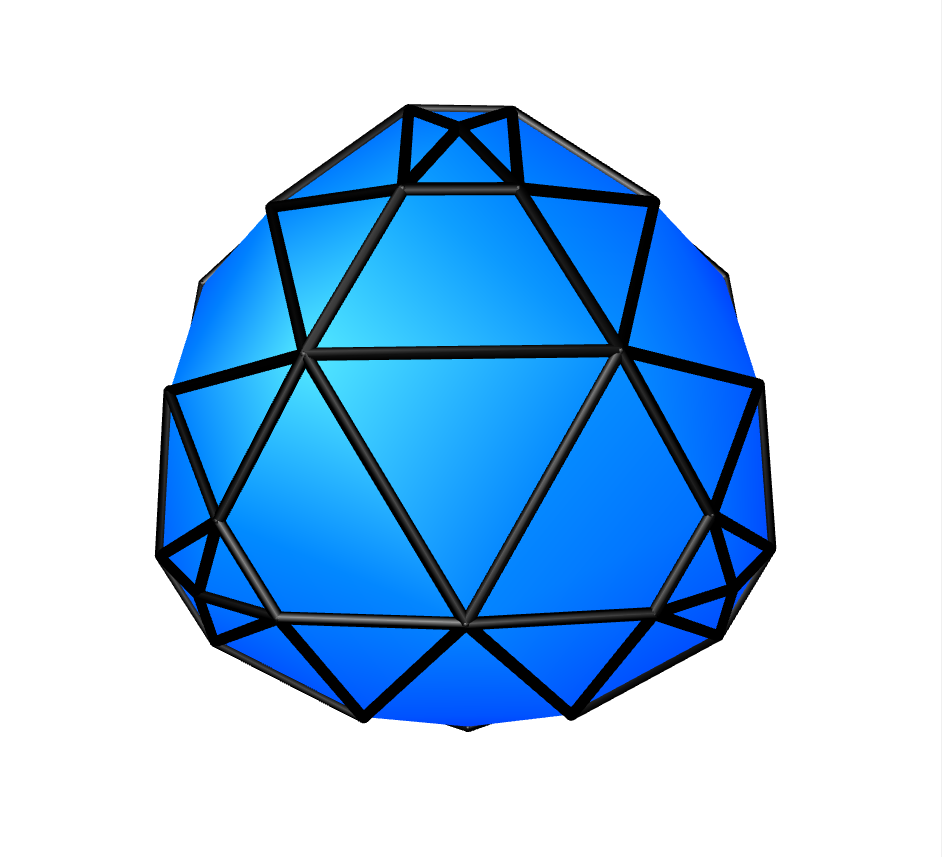}
\endminipage\hfill
\minipage{0.32\textwidth}
  \includegraphics[width=3.7cm]{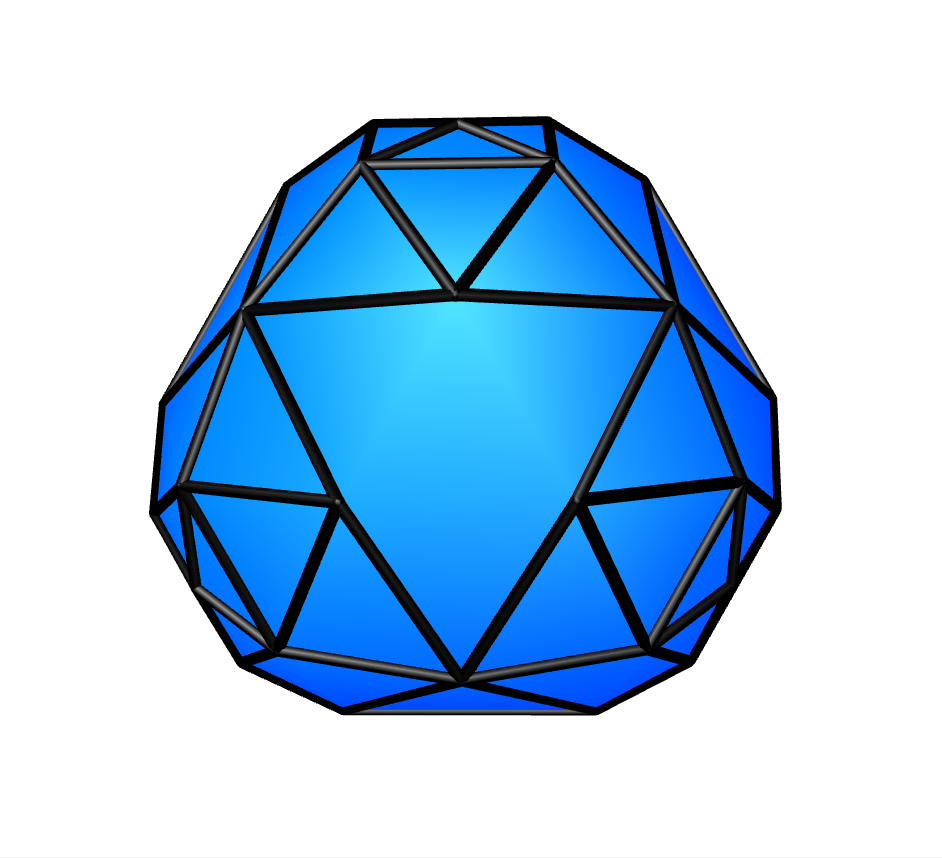}
\endminipage\hfill
\caption{Polyhedron of level 1 (top) and level 2 (bottom)}
\label{fig:3d}
\end{figure}
\section{FEM}
To implement the finite element method, we need a Gram matrix and an energy matrix which are sparse matrices filled with contributions from areas of triangles for each of the vertices. We compute the beginning of the spectrum for different levels and increasing subdivisions.
\subsection{Eigenvalues}
We approximate the spectrum of the Laplacian on the surface by extrapolating over the spectrum of lower level polyhedra. The Laplacian on the surface is just the usual two dimensional Laplacian $\Delta = \frac{\partial^2}{\partial x^2} + \frac{\partial^2}{\partial y^2}$ on the planar realization, such that values and normal derivatives along the identified edges match. By the spectrum we mean a study of both the eigenvalues $\lambda$ and and eigenfunctions $u$ satisfying 
\begin{equation}
\label{eq:cond}
-\Delta u = \lambda u\text{.}
\end{equation}
It is known that the eigenvalues form an increasing sequence $0 = \lambda_0 < \lambda_1 \leq \lambda_2 \leq \ldots$ tending to infinity, and satisfying the Weyl asymptotic law
\begin{equation}
\label{eq:weyl}
N(t) = \#\{\lambda_j \leq t\} \sim \frac{A}{4\pi}t
\end{equation}
where $A$ is the area of the surface. 
In Figures \ref{fig:cf1}-\ref{fig:cf3} we show the graphs of the eigenvalue counting function $N(t)$, the difference
\begin{equation}
D(t) = N(t) -  \frac{A}{4\pi}t\text{,}
\end{equation}
and the average of the difference
\begin{equation}
A(t) = \frac{1}{t}\int_0^t D(s)\text{ ds}
\end{equation}
based on our numerical computations.

\clearpage

\begin{figure}[H]
\includegraphics[width=\linewidth]{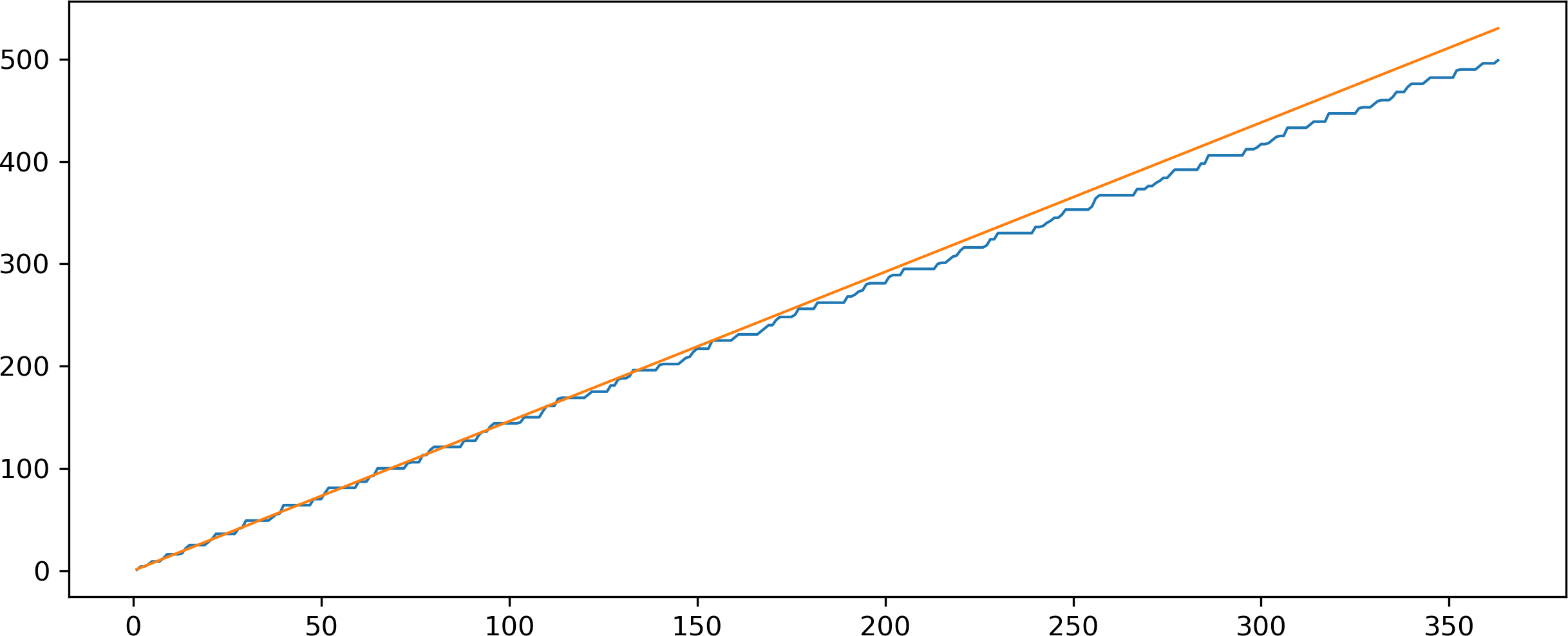}
\vspace{-0.3cm}
\caption{The eigenvalue counting function (blue) and the Weyl term (orange)}
\label{fig:cf1}
\end{figure}

\begin{figure}[H]
\includegraphics[width=\linewidth]{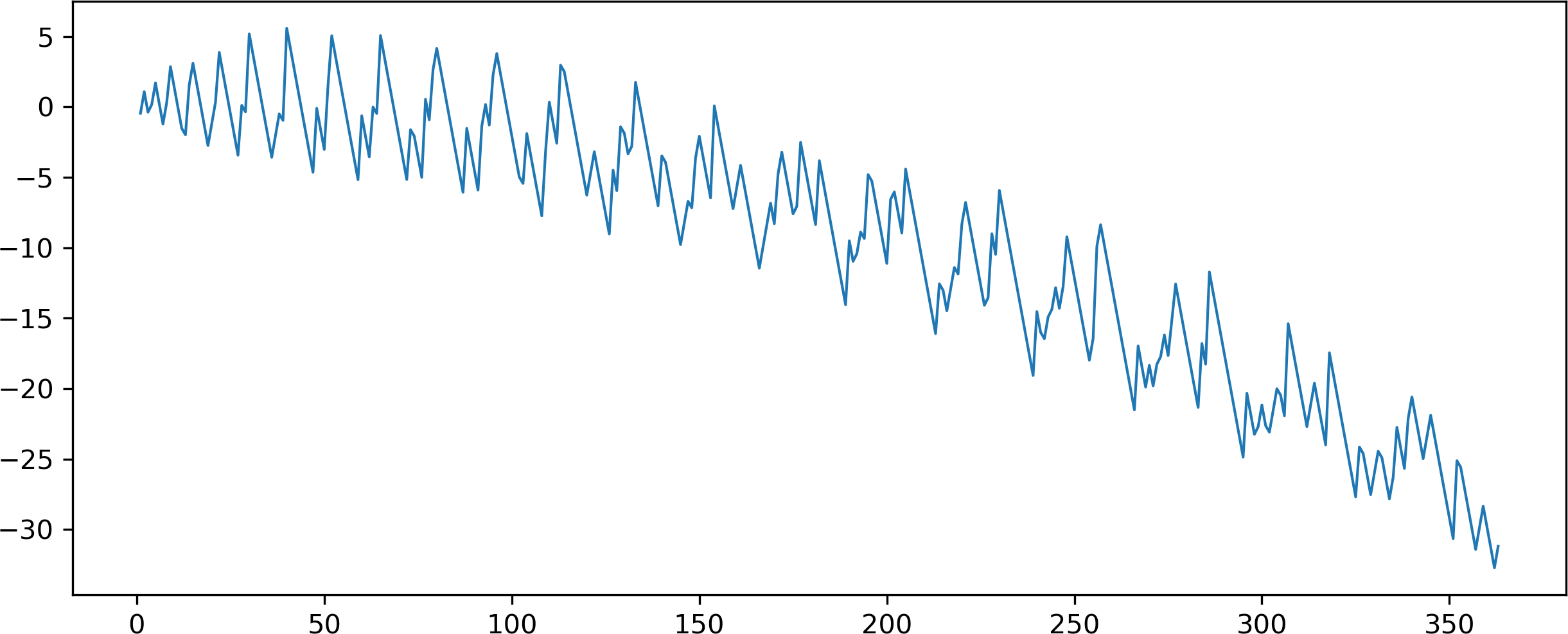}
\vspace{-0.3cm}
\caption{The difference $D(t)$}
\label{fig:cf2}
\end{figure}

\begin{figure}[H]
\includegraphics[width=\linewidth]{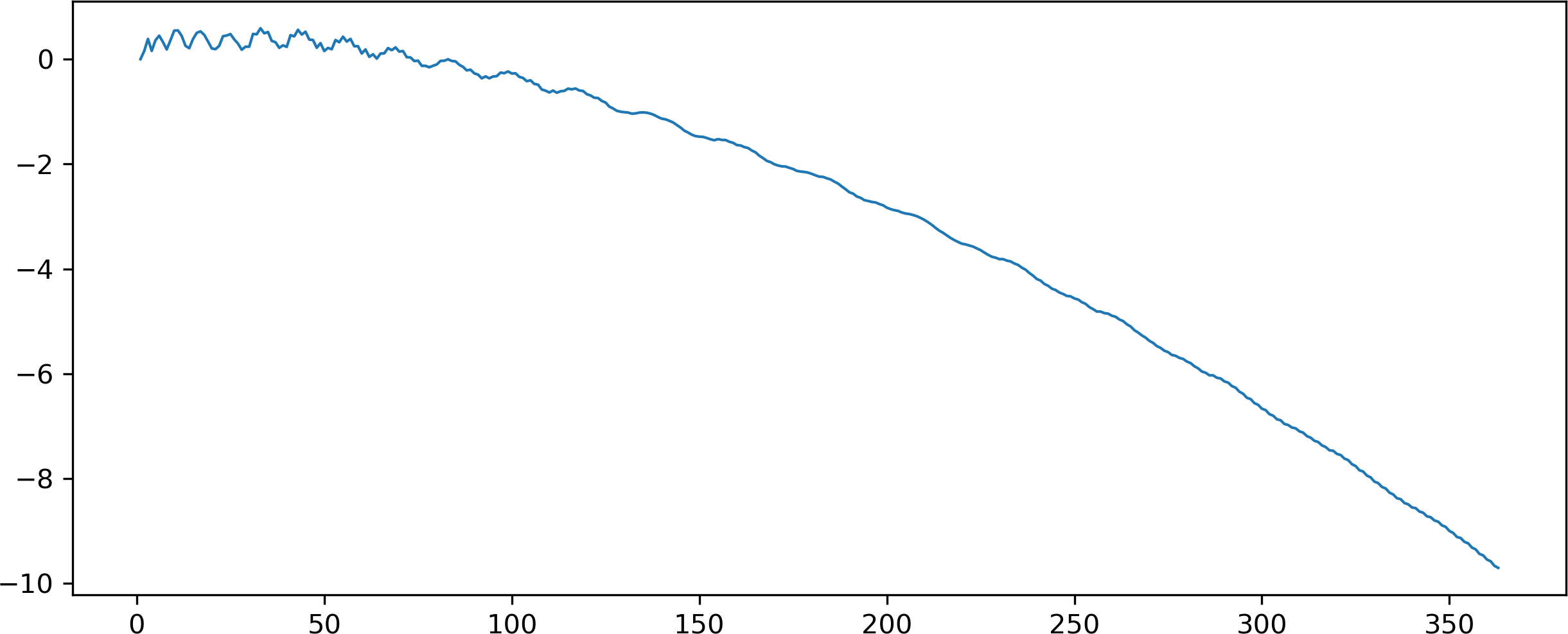}
\vspace{-0.3cm}
\caption{The averaged distance $A(t)$}
\label{fig:cf3}
\end{figure}

\clearpage

The multiplicities at the beginning of the spectrum of level are shown in Figure \ref{fig:gr2}. They are distributed as follows: $\frac{1}{12}$ of eigenvalues have multiplicity 1, $\frac{1}{6}$ have multiplicity 2 and $\frac{3}{4}$ have multiplicity 3. This corresponds to the squares of dimensions of irreducible representations (1, 2 and 3) of the symmetry group $S_4$.

We notice a pattern of multiplicities around consecutive multiplicity 1 occurances (3-3-1-1-3) starting at positions 9, 38, 105, 198... and found that for the corresponding pair of eigenfunctions the individual faces are either invariant under all symmetries or just under one (skew)symmetric reflection. 
\begin{figure}[H]
\includegraphics[width=\linewidth]{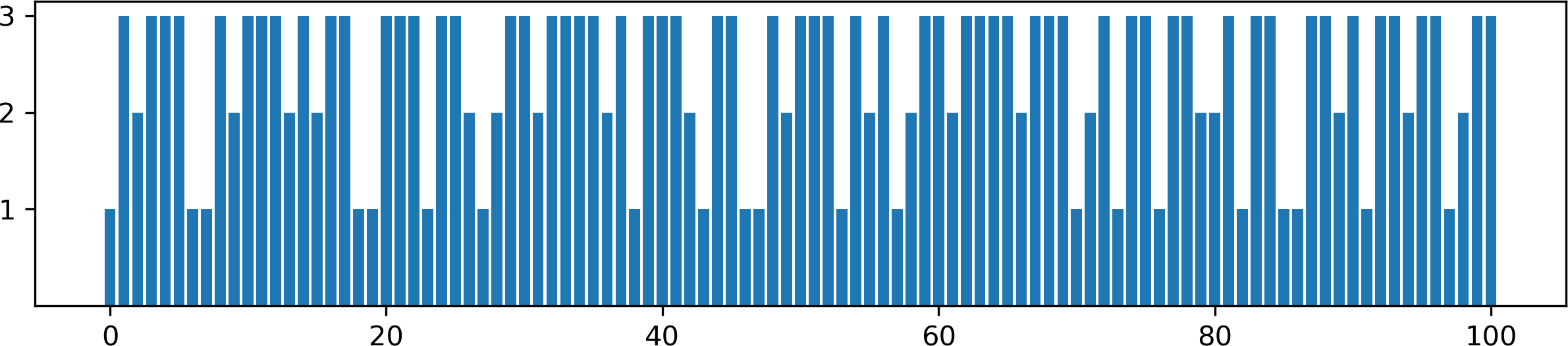}
\vspace{-0.3cm}
\caption{Multiplicities at the beginning of the spectrum of level 7}
\label{fig:gr2}
\end{figure}

We extrapolate exponentially using levels 4 to 7. Figure \ref{fig:gr1} shows the beginning of the spectrum for the first levels. 
\begin{figure}[H]
\includegraphics[width=\linewidth]{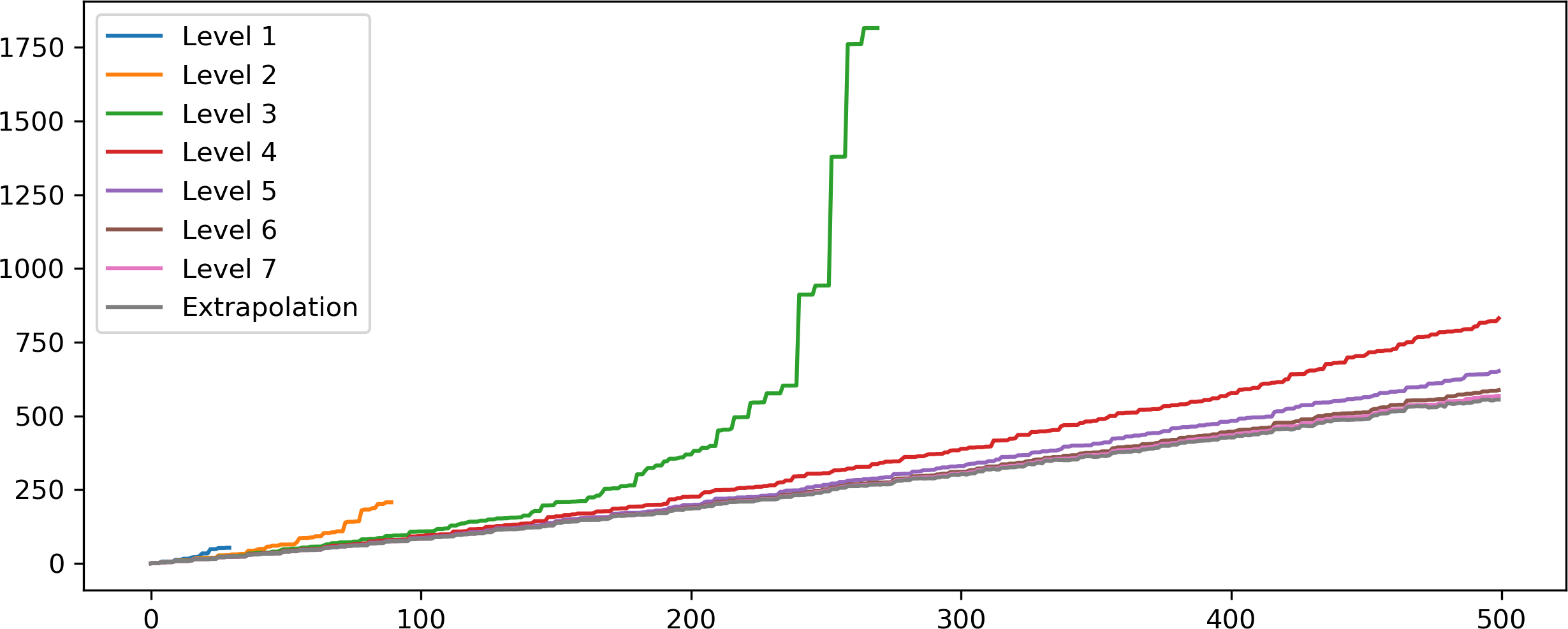}
\vspace{-0.3cm}
\caption{Beginning of the spectrum}
\label{fig:gr1}
\end{figure}
Tables 1-4 show the beginning of the spectrum for level 1 to 7 without subdivisions. The numbering in the left column refers to Level 7. Note that below the double-horizontal line on the next page the multiplicities of level $6$ and $7$ do not match any more.

\newpage

\begin{table}[H]
\centering
\begin{tabular}{|c|c|c|c|c|c|} 
\hline
\#      & Level 4    & Level 5    & Level 6    & Level 7    & Extrap.  \\ 
\hline
0       & 0 (1)      & 0 (1)      & 0 (1)      & 0 (1)      & 0.00    \\
1-3     & 1.37 (3)   & 1.37 (3)   & 1.37 (3)   & 1.37 (3)   & 1.37    \\
4-5     & 3.93 (2)   & 3.92 (2)   & 3.91 (2)   & 3.91 (2)   & 3.91    \\
6-8     & 4.37 (3)   & 4.35 (3)   & 4.34 (3)   & 4.34 (3)   & 4.34    \\
9-11    & 8.38 (3)   & 8.3 (3)    & 8.28 (3)   & 8.27 (3)   & 8.26    \\
12-14   & 8.56 (3)   & 8.49 (3)   & 8.48 (3)   & 8.47 (3)   & 8.47    \\
15      & 9.41 (1)   & 9.34 (1)   & 9.31 (1)   & 9.3 (1)    & 9.30    \\
16      & 13.68 (1)  & 13.5 (1)   & 13.44 (1)  & 13.42 (1)  & 13.42   \\
17-19   & 14.53 (2)  & 14.35 (2)  & 14.28 (3)  & 14.24 (3)  & 14.23   \\
20-21   & 14.62 (3)  & 14.36 (3)  & 14.3 (2)   & 14.28 (2)  & 14.27   \\
22-24   & 15.47 (3)  & 15.28 (3)  & 15.23 (3)  & 15.21 (3)  & 15.20   \\
25-27   & 21.33 (3)  & 20.9 (3)   & 20.77 (3)  & 20.72 (3)  & 20.71   \\
28-30   & 23.14 (3)  & 22.54 (3)  & 22.33 (3)  & 22.25 (3)  & 22.21   \\
31-32   & 23.24 (2)  & 22.93 (2)  & 22.86 (2)  & 22.85 (2)  & 22.84   \\
33-35   & 24.2 (3)   & 23.65 (3)  & 23.48 (3)  & 23.42 (3)  & 23.40   \\
36-37   & 31.14 (2)  & 29.95 (2)  & 29.56 (2)  & 29.41 (2)  & 29.35   \\
38-40   & 31.4 (3)   & 30.66 (3)  & 30.42 (3)  & 30.34 (3)  & 30.30   \\
41-43   & 33.69 (3)  & 32.88 (3)  & 32.64 (3)  & 32.57 (3)  & 32.55   \\
44      & 33.98 (1)  & 33.1 (1)   & 32.85 (1)  & 32.77 (1)  & 32.74   \\
45      & 35.47 (3)  & 34.04 (1)  & 33.37 (1)  & 33.1 (1)   & 33.04   \\
46-48   & 36.02 (1)  & 34.4 (3)   & 34.12 (3)  & 34.03 (3)  & 34.00   \\
49-51   & 42.5 (3)   & 40.83 (3)  & 40.28 (3)  & 40.08 (3)  & 39.99   \\
52-54   & 43.74 (3)  & 41.97 (3)  & 41.4 (3)   & 41.2 (3)   & 41.11   \\
55      & 46.14 (1)  & 44.87 (1)  & 44.53 (1)  & 44.44 (1)  & 44.41   \\
56-58   & 47.63 (3)  & 45.97 (3)  & 45.32 (3)  & 45.09 (3)  & 44.96   \\
59-61   & 48.13 (3)  & 46.0 (3)   & 45.56 (3)  & 45.43 (3)  & 45.42   \\
62-63   & 49.39 (2)  & 47.4 (2)   & 46.83 (2)  & 46.63 (2)  & 46.56   \\
64      & 54.45 (1)  & 52.24 (1)  & 51.52 (1)  & 51.26 (1)  & 51.14   \\
65-66   & 56.42 (2)  & 54.0 (2)   & 53.24 (2)  & 52.97 (2)  & 52.86   \\
67-69   & 59.78 (3)  & 56.0 (3)   & 54.77 (3)  & 54.3 (3)   & 54.09   \\
70-72   & 60.65 (3)  & 58.25 (3)  & 57.59 (3)  & 57.39 (3)  & 57.33   \\
73-74   & 62.74 (2)  & 60.31 (2)  & 59.67 (2)  & 59.5 (2)   & 59.44   \\
75-77   & 64.76 (3)  & 61.66 (3)  & 60.76 (3)  & 60.47 (3)  & 60.36   \\
78-80   & 67.09 (3)  & 63.17 (3)  & 61.94 (3)  & 61.49 (3)  & 61.31   \\
81-83   & 74.27 (3)  & 69.87 (3)  & 68.4 (3)   & 67.87 (3)  & 67.62   \\
84-86   & 76.3 (3)   & 71.39 (3)  & 69.77 (3)  & 69.18 (3)  & 68.91   \\
87-88   & 79.09 (2)  & 75.35 (2)  & 74.21 (2)  & 73.84 (2)  & 73.69   \\
89-91   & 79.66 (3)  & 76.28 (3)  & 75.46 (3)  & 75.29 (3)  & 75.22   \\
92      & 79.76 (1)  & 76.64 (1)  & 75.77 (1)  & 75.56 (1)  & 75.47   \\
93-95   & 81.06 (3)  & 77.1 (3)   & 76.0 (3)   & 75.68 (3)  & 75.56   \\
96-98   & 91.32 (3)  & 85.35 (3)  & 83.34 (3)  & 82.62 (3)  & 82.27   \\
99-101  & 92.9 (3)   & 86.03 (3)  & 84.23 (3)  & 83.7 (3)   & 83.53   \\
102-103 & 95.67 (2)  & 87.53 (2)  & 84.91 (2)  & 83.94 (2)  & 83.52   \\
104     & 96.88 (3)  & 88.62 (1)  & 85.55 (1)  & 84.41 (1)  & 83.78   \\
105-107 & 98.31 (1)  & 91.31 (3)  & 89.69 (3)  & 89.2 (3)   & 89.00   \\
108-110 & 98.48 (3)  & 92.63 (3)  & 91.11 (3)  & 90.72 (3)  & 90.58   \\ \hhline{|=|=|=|=|=|=|}
111     & 98.52 (1)  & 94.12 (1)  & 92.85 (2)  & 91.65 (1)  & 91.58   \\
112     & 107.97 (1) & 96.25 (1)  & 97.62 (3)  & 92.49 (1)  & 92.05   \\
\hline
\end{tabular}
\vspace{0.2cm}
\caption{Beginning of the spectrum for different level $m$ and the extrapolation}
\end{table}

\begin{table}[H]
\centering
\begin{tabular}{|c|c|c|c|c|c|} 
\hline
\#      & Level 4    & Level 5    & Level 6    & Level 7    & Extrap.  \\ 
\hline
113-115 & 108.08 (3) & 99.98 (3)  & 99.45 (2)  & 96.85 (3)  & 96.57   \\
116-117 & 109.2 (2)  & 101.54 (2) & 102.01 (3) & 98.83 (2)  & 98.61   \\
118-120 & 115.67 (3) & 105.42 (3) & 103.18 (3) & 100.8 (3)  & 100.22  \\
121-123 & 116.6 (3)  & 106.23 (3) & 107.44 (3) & 102.13 (3) & 101.76  \\
124-126 & 123.37 (1) & 111.19 (3) & 112.34 (1) & 106.15 (3) & 105.63  \\
127     & 123.85 (3) & 114.66 (1) & 115.74 (3) & 111.62 (1) & 111.43  \\
128-130 & 126.96 (3) & 118.98 (3) & 116.93 (2) & 114.69 (3) & 114.07  \\
131-132 & 129.26 (3) & 119.19 (3) & 117.02 (3) & 115.66 (2) & 115.60  \\
133-135 & 130.66 (1) & 120.7 (2)  & 118.45 (1) & 116.42 (3) & 116.31  \\
136     & 130.99 (2) & 121.55 (1) & 119.17 (2) & 117.44 (1) & 116.94  \\
137-138 & 134.71 (2) & 122.05 (2) & 122.73 (3) & 118.26 (2) & 118.11  \\
139-141 & 135.25 (3) & 125.54 (3) & 124.19 (3) & 121.89 (3) & 121.56  \\
142-144 & 143.13 (2) & 128.47 (3) & 127.58 (2) & 122.75 (3) & 122.23  \\
145-146 & 143.36 (3) & 131.36 (2) & 130.9 (3)  & 126.36 (2) & 125.81  \\
147-149 & 157.02 (3) & 136.57 (3) & 139.19 (3) & 128.95 (3) & 128.35  \\
150-152 & 159.47 (3) & 143.73 (3) & 144.08 (3) & 137.76 (3) & 137.22  \\
153-155 & 163.81 (3) & 148.98 (3) & 144.56 (3) & 141.95 (3) & 141.07  \\
156-158 & 166.1 (2)  & 150.23 (3) & 148.93 (2) & 142.95 (3) & 141.90  \\
159-160 & 169.13 (3) & 152.91 (2) & 149.55 (3) & 147.78 (2) & 147.48  \\
161-163 & 169.22 (3) & 154.06 (3) & 150.52 (3) & 148.11 (3) & 147.54  \\
164-166 & 170.58 (1) & 156.56 (3) & 151.88 (3) & 148.47 (3) & 147.52  \\
167-169 & 176.02 (3) & 156.8 (3)  & 154.91 (1) & 150.72 (3) & 150.28  \\
170     & 176.65 (3) & 158.31 (1) & 162.94 (2) & 154.11 (1) & 154.00  \\
171-172 & 185.4 (3)  & 168.55 (3) & 163.57 (3) & 160.66 (2) & 159.67  \\
173-175 & 186.44 (3) & 170.0 (2)  & 165.48 (3) & 162.11 (3) & 161.49  \\
176     & 192.37 (3) & 171.27 (3) & 166.23 (1) & 163.48 (1) & 162.28  \\
177-179 & 192.78 (2) & 171.28 (3) & 166.93 (3) & 163.8 (3)  & 163.21  \\
180-182 & 193.05 (1) & 173.2 (1)  & 167.66 (1) & 165.77 (3) & 165.61  \\
183     & 196.51 (1) & 174.11 (1) & 169.02 (3) & 166.12 (1) & 165.36  \\
184-186 & 198.24 (3) & 176.18 (3) & 173.13 (3) & 166.73 (3) & 165.61  \\
187-189 & 198.73 (3) & 179.28 (3) & 174.77 (2) & 171.33 (3) & 170.44  \\
190-191 & 201.58 (2) & 181.95 (2) & 180.61 (2) & 172.04 (2) & 170.53  \\
192-193 & 217.16 (3) & 188.64 (2) & 185.07 (3) & 178.51 (2) & 177.62  \\
194-196 & 221.2 (2)  & 192.53 (3) & 187.06 (1) & 182.85 (3) & 181.87  \\
197     & 225.43 (3) & 197.73 (3) & 189.17 (3) & 183.24 (1) & 181.14  \\
198-200 & 226.05 (3) & 197.93 (1) & 190.58 (3) & 186.46 (3) & 185.27  \\
201-203 & 226.14 (1) & 198.84 (3) & 194.71 (1) & 187.93 (3) & 186.83  \\
204     & 235.29 (1) & 201.4 (1)  & 196.43 (1) & 192.62 (1) & 192.46  \\
205     & 241.4 (3)  & 207.81 (1) & 199.31 (3) & 192.86 (1) & 190.92  \\
206-208 & 244.76 (1) & 209.82 (3) & 207.73 (3) & 195.86 (3) & 194.13  \\
209-211 & 248.45 (3) & 219.27 (2) & 208.8 (2)  & 203.62 (3) & 201.47  \\
212-213 & 249.14 (3) & 219.3 (3)  & 212.48 (3) & 205.39 (2) & 203.43  \\
214-216 & 249.79 (2) & 220.7 (3)  & 213.16 (3) & 208.68 (3) & 207.65  \\
217     & 254.04 (1) & 223.52 (1) & 213.48 (1) & 210.22 (1) & 208.48  \\
218-220 & 255.38 (3) & 223.54 (3) & 214.29 (3) & 210.97 (3) & 209.41  \\
221-223 & 257.26 (3) & 224.89 (3) & 216.61 (2) & 211.32 (3) & 209.67  \\
224-225 & 259.79 (3) & 225.39 (2) & 220.61 (3) & 214.18 (2) & 213.43  \\
226-228 & 261.87 (2) & 229.5 (3)  & 221.13 (3) & 217.85 (3) & 216.76  \\
\hline
\end{tabular}
\vspace{0.2cm}
\caption{Beginning of the spectrum for different level $m$ and the extrapolation}
\end{table}

\begin{table}[H]
\centering
\begin{tabular}{|c|c|c|c|c|c|} 
\hline
\#      & Level 4    & Level 5    & Level 6    & Level 7    & Extrap.  \\ 
\hline
229-231 & 264.86 (3) & 230.81 (3) & 222.42 (1) & 218.03 (3) & 216.94  \\
232     & 272.86 (1) & 233.29 (1) & 230.83 (2) & 219.63 (1) & 218.49  \\
233-234 & 274.65 (2) & 242.37 (2) & 232.62 (3) & 226.85 (2) & 224.58  \\
235-237 & 281.02 (3) & 246.9 (3)  & 235.64 (3) & 227.85 (3) & 225.39  \\
238-240 & 294.81 (2) & 247.35 (3) & 237.53 (2) & 232.05 (3) & 231.18  \\
241-242 & 295.55 (3) & 251.22 (2) & 242.32 (3) & 233.58 (2) & 231.71  \\
243-245 & 303.96 (3) & 257.87 (3) & 246.39 (3) & 236.83 (3) & 234.13  \\
246-248 & 304.15 (3) & 261.82 (3) & 251.53 (3) & 241.68 (3) & 238.64  \\
249-251 & 305.81 (3) & 266.33 (3) & 257.87 (3) & 246.83 (3) & 243.71  \\
252-254 & 310.2 (1)  & 270.84 (3) & 262.7 (3)  & 253.11 (3) & 250.95  \\
255-257 & 315.55 (3) & 276.12 (3) & 267.68 (2) & 258.49 (3) & 256.18  \\
258-259 & 317.26 (2) & 280.48 (2) & 268.27 (3) & 263.42 (2) & 261.59  \\
260-262 & 321.5 (3)  & 281.96 (1) & 270.92 (3) & 264.13 (3) & 261.71  \\
263-265 & 326.86 (3) & 282.89 (3) & 273.58 (1) & 265.73 (3) & 264.53  \\
266-268 & 327.19 (3) & 285.66 (3) & 273.72 (3) & 268.41 (3) & 266.90  \\
269-270 & 336.55 (3) & 287.59 (3) & 274.02 (2) & 269.02 (2) & 267.47  \\
271-273 & 341.09 (2) & 290.22 (2) & 274.57 (3) & 270.65 (3) & 268.31  \\
274     & 344.79 (3) & 292.16 (3) & 283.35 (1) & 271.51 (1) & 268.74  \\
275     & 345.82 (3) & 302.76 (3) & 287.54 (3) & 276.56 (1) & 273.96  \\
276-278 & 353.11 (1) & 303.45 (1) & 288.09 (3) & 282.66 (3) & 279.81  \\
279-281 & 360.97 (3) & 303.78 (3) & 295.23 (3) & 284.2 (3)  & 282.55  \\
282-284 & 360.99 (2) & 311.17 (3) & 295.96 (1) & 289.99 (3) & 287.58  \\
285     & 363.24 (3) & 312.4 (1)  & 296.93 (3) & 291.09 (1) & 288.59  \\
286-287 & 366.82 (1) & 315.92 (3) & 297.99 (2) & 291.65 (2) & 287.48  \\
288-290 & 370.24 (3) & 317.18 (2) & 301.67 (2) & 291.78 (3) & 288.43  \\
291-292 & 371.34 (3) & 322.13 (2) & 304.11 (3) & 295.19 (2) & 291.42  \\
293-295 & 376.36 (2) & 325.33 (3) & 310.04 (3) & 297.44 (3) & 294.41  \\
296     & 382.34 (1) & 328.38 (3) & 310.61 (3) & 302.76 (1) & 299.38  \\
297-299 & 383.35 (3) & 330.1 (3)  & 311.07 (1) & 303.7 (3)  & 300.61  \\
300-302 & 388.63 (3) & 335.19 (1) & 315.3 (1)  & 305.67 (3) & 302.57  \\
303     & 392.67 (3) & 336.88 (1) & 316.4 (2)  & 308.67 (1) & 303.89  \\
304-305 & 394.0 (3)  & 338.1 (2)  & 321.72 (3) & 309.82 (2) & 305.24  \\
306-308 & 395.84 (2) & 341.85 (1) & 325.79 (1) & 316.28 (3) & 312.01  \\
309     & 396.62 (1) & 341.87 (3) & 328.41 (3) & 321.57 (1) & 319.54  \\
310-311 & 416.57 (3) & 346.51 (3) & 328.87 (2) & 321.95 (2) & 318.15  \\
312-314 & 416.99 (3) & 351.96 (2) & 335.88 (3) & 322.38 (3) & 320.85  \\
315-317 & 422.45 (2) & 360.9 (3)  & 338.16 (3) & 328.08 (3) & 324.14  \\
318-320 & 424.11 (1) & 361.41 (3) & 341.41 (3) & 330.71 (3) & 325.64  \\
321-323 & 435.62 (2) & 366.5 (2)  & 345.84 (1) & 333.3 (3)  & 328.33  \\
324     & 435.77 (3) & 367.16 (3) & 347.11 (2) & 336.65 (1) & 333.41  \\
325     & 445.38 (3) & 374.47 (1) & 350.8 (1)  & 340.76 (1) & 338.34  \\
326-327 & 447.57 (3) & 376.51 (3) & 352.45 (3) & 341.27 (2) & 335.74  \\
328-330 & 449.1 (1)  & 379.87 (3) & 357.95 (3) & 345.02 (3) & 340.66  \\
331-333 & 450.18 (1) & 382.46 (3) & 360.47 (1) & 351.64 (3) & 348.31  \\
334     & 452.25 (3) & 384.72 (1) & 360.5 (3)  & 352.39 (1) & 349.10  \\
335-337 & 465.24 (1) & 387.13 (1) & 364.91 (3) & 354.74 (3) & 351.62  \\
338-340 & 468.7 (3)  & 396.03 (3) & 366.64 (3) & 355.08 (3) & 349.72  \\
341-343 & 469.18 (3) & 399.26 (2) & 372.59 (2) & 357.03 (3) & 352.18  \\
\hline
\end{tabular}
\vspace{0.2cm}
\caption{Beginning of the spectrum for different level $m$ and the extrapolation}
\end{table}

\begin{table}[H]
\centering
\begin{tabular}{|c|c|c|c|c|c|} 
\hline
\#      & Level 4    & Level 5    & Level 6    & Level 7    & Extrap.  \\ 
\hline
344-345 & 475.76 (2) & 399.32 (3) & 374.47 (3) & 364.75 (2) & 359.96  \\
346-348 & 481.12 (3) & 400.07 (3) & 376.35 (3) & 366.73 (3) & 363.02  \\
349-351 & 483.09 (2) & 405.95 (3) & 379.65 (3) & 367.49 (3) & 361.10  \\
352-354 & 489.5 (3)  & 407.43 (1) & 379.67 (1) & 370.32 (3) & 365.52  \\
355     & 493.52 (1) & 410.16 (3) & 390.25 (2) & 372.37 (1) & 367.35  \\
356-357 & 499.79 (3) & 422.95 (2) & 393.67 (3) & 380.68 (2) & 374.58  \\
358-360 & 509.36 (3) & 424.44 (3) & 395.82 (3) & 383.73 (3) & 377.65  \\
361-363 & 510.71 (2) & 430.47 (3) & 396.96 (3) & 384.32 (3) & 378.77  \\
364-366 & 511.2 (3)  & 433.33 (3) & 404.26 (3) & 386.75 (3) & 381.10  \\
367-369 & 521.12 (3) & 435.62 (2) & 405.4 (2)  & 392.64 (3) & 385.94  \\
370-371 & 521.55 (2) & 440.62 (2) & 407.96 (2) & 395.13 (2) & 388.87  \\
372-373 & 523.18 (3) & 441.97 (3) & 414.76 (1) & 398.75 (2) & 391.79  \\
374     & 526.68 (1) & 444.89 (1) & 416.44 (3) & 404.86 (1) & 398.66  \\
375-377 & 533.4 (3)  & 448.71 (3) & 418.93 (3) & 405.36 (3) & 398.59  \\
378-380 & 536.16 (3) & 458.1 (3)  & 425.65 (3) & 407.83 (3) & 402.84  \\
381-383 & 539.51 (3) & 459.87 (1) & 428.42 (3) & 414.1 (3)  & 408.64  \\
384-386 & 541.85 (1) & 462.15 (3) & 430.28 (1) & 419.44 (3) & 412.61  \\
387-388 & 547.75 (3) & 463.65 (3) & 432.0 (2)  & 421.21 (2) & 413.75  \\
389     & 549.18 (2) & 467.98 (3) & 433.65 (3) & 422.46 (1) & 415.84  \\
390-392 & 553.87 (3) & 470.79 (2) & 437.44 (3) & 423.35 (3) & 416.14  \\
393-395 & 559.46 (3) & 472.57 (3) & 444.03 (3) & 426.76 (3) & 419.63  \\
396-398 & 567.23 (3) & 480.13 (3) & 446.33 (3) & 432.79 (3) & 425.99  \\
399-401 & 574.44 (1) & 480.56 (1) & 449.68 (1) & 434.37 (3) & 427.32  \\
402-404 & 577.38 (3) & 483.45 (3) & 452.66 (2) & 439.61 (3) & 433.19  \\
405-407 & 588.63 (2) & 491.07 (3) & 453.22 (3) & 443.02 (3) & 435.24  \\
408-410 & 590.79 (3) & 493.44 (2) & 456.9 (3)  & 445.54 (3) & 437.45  \\
411-413 & 594.91 (3) & 494.94 (3) & 459.47 (3) & 447.86 (3) & 442.29  \\
414-415 & 606.12 (1) & 495.35 (2) & 460.8 (2)  & 450.21 (2) & 444.31  \\
416-418 & 609.45 (3) & 497.8 (3)  & 475.84 (1) & 461.66 (3) & 454.48  \\
419-421 & 612.36 (2) & 513.83 (1) & 476.74 (3) & 463.69 (3) & 455.97  \\
422     & 614.61 (3) & 515.77 (3) & 476.82 (3) & 463.91 (1) & 453.68  \\
423     & 623.89 (2) & 522.22 (1) & 478.85 (1) & 464.65 (1) & 454.90  \\
424-425 & 641.18 (3) & 524.61 (3) & 481.9 (2)  & 465.9 (2)  & 458.41  \\
426-427 & 641.7 (3)  & 530.62 (2) & 488.24 (2) & 474.77 (2) & 466.83  \\
428-430 & 650.4 (1)  & 536.23 (2) & 488.48 (3) & 474.9 (3)  & 465.56  \\
431     & 653.88 (3) & 536.69 (3) & 493.48 (1) & 477.32 (1) & 470.41  \\
432-434 & 657.49 (1) & 544.86 (3) & 498.91 (3) & 486.25 (3) & 476.44  \\
435-437 & 659.13 (2) & 545.65 (3) & 502.99 (3) & 488.8 (3)  & 482.00  \\
438-440 & 676.14 (3) & 547.0 (1)  & 506.87 (3) & 494.09 (3) & 486.49  \\
441-443 & 679.73 (2) & 551.04 (3) & 508.39 (2) & 494.58 (3) & 487.07  \\
444-445 & 681.0 (3)  & 552.17 (2) & 508.86 (3) & 495.16 (2) & 486.99  \\
446-448 & 698.25 (3) & 557.04 (3) & 510.66 (3) & 496.13 (3) & 487.99  \\
449-451 & 702.54 (1) & 559.13 (3) & 512.14 (3) & 498.7 (3)  & 488.65  \\
452     & 702.84 (3) & 563.66 (3) & 521.19 (1) & 504.46 (1) & 498.38  \\
453-454 & 708.14 (1) & 566.38 (1) & 523.5 (2)  & 509.05 (2) & 502.00  \\
455-457 & 715.85 (3) & 570.25 (2) & 527.58 (1) & 515.94 (3) & 507.25  \\
458     & 719.75 (3) & 577.52 (3) & 529.2 (3)  & 517.35 (1) & 509.18  \\
459-461 & 722.16 (3) & 578.34 (1) & 536.85 (3) & 521.81 (3) & 514.89  \\
\hline
\end{tabular}
\vspace{0.2cm}
\caption{Beginning of the spectrum for different level $m$ and the extrapolation}
\end{table}
\clearpage
\subsection{Eigenfunctions}
We display the eigenfunctions as described above in Fig. \ref{fig:all}, and study the symmetries at the two selected tops ($x>0$ and $x<0$) locally, and their relation to each other. The eigenfunctions may be sorted into symmetry types according to the irreducible representations of the symmetry group of the polyhedron. We will show examples of the eigenfunction symmetry types for level $1$ and $2$. Note, that the types of symmetries that occur in the eigenfunctions correspond to the irreducible representations of the symmetry group of the polyhedron, $S_4$, which has two 1-dim. representations, one 2-dim. representation and two 3-dim. representations.

There are two types of eigenfunctions of multiplicity 1, since there are two 1-dim. irreducible representations of $S_4$. One type is symmetrical with regard to both axes (Fig. \ref{fig:ex1}), the other has horizontal symmetry and is skew-symmetrical vertically or the other way around (Fig. \ref{fig:ex2}). This is equivalent to saying that one type has rotational symmetry and the other has rotational skew-symmetry by $\pi$ with regard to the origin.
\begin{figure}[H]
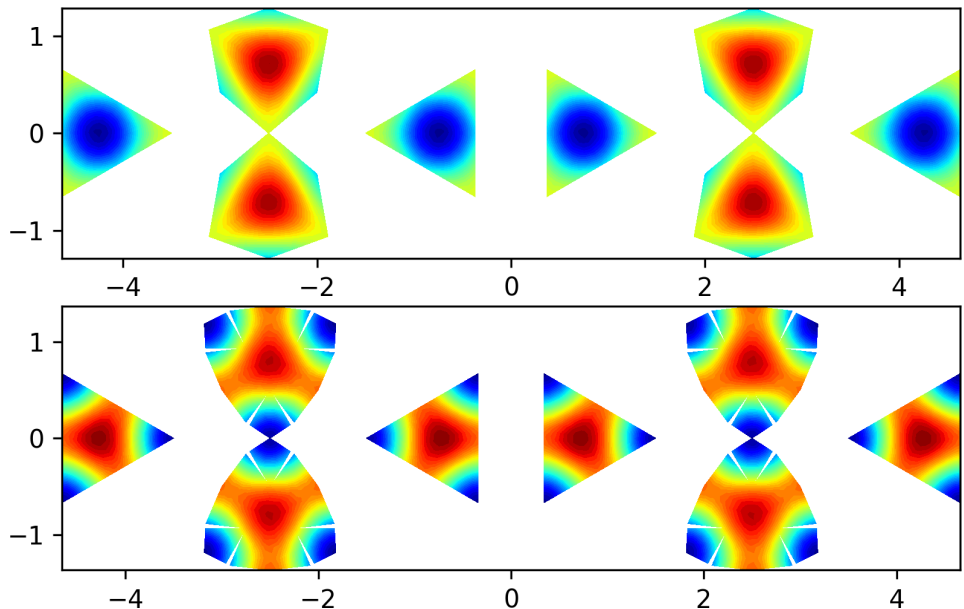

\minipage{\textwidth}
\centering
  \includegraphics[width=\linewidth]{fig/Pics_1_50/EF_1_15.png}
\endminipage
\vspace{0.1cm}
\minipage{\textwidth}
\centering
  \includegraphics[width=\linewidth]{fig/Pics_2_50/EF_2_16.png}
\endminipage\hfill
\caption{Eigenfunctions 15 ($m=1$) and 16 ($m=2$)}
\label{fig:ex1}
\end{figure}

\clearpage

\begin{figure}[H]
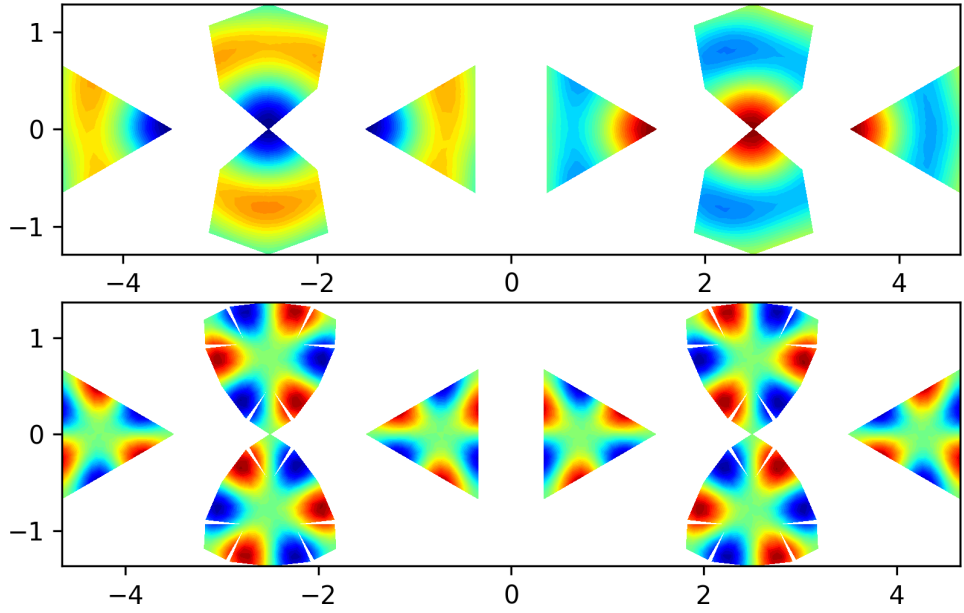

\minipage{\textwidth}
\centering
  \includegraphics[width=\linewidth]{fig/Pics_1_50/EF_1_14.png}
\endminipage
\vspace{0.1cm}
\minipage{\textwidth}
\centering
  \includegraphics[width=\linewidth]{fig/Pics_2_50/EF_2_41.png}
\endminipage\hfill
\caption{Eigenfunctions 14 ($m=1$) and 41 ($m=2$)}
\label{fig:ex2}
\end{figure}

\clearpage

There is one type of pairs of eigenfunctions from a 2-dimensional eigenspace. Both eigenfunctions from such a pair have horizontal symmetry. Individually, they seem to have rotational symmetry by $\frac{\pi}{2}$ between the tops. One eigenfunction in the 2-dimensional eigenspace is obtained from the other by rotation of $\pi$ and inverting on half. Two pairs for level $1$ and $2$ are shown in Fig. \ref{fig:ex3} and Fig. \ref{fig:ex4}.
\begin{figure}[H]
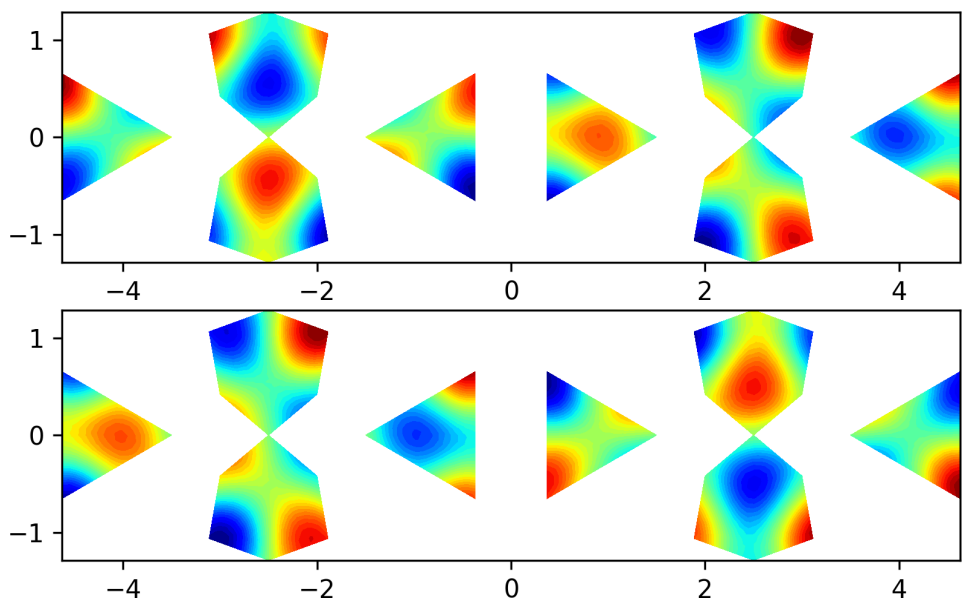

\minipage{\textwidth}
\centering
  \includegraphics[width=\linewidth]{fig/Pics_1_50/EF_1_12.png}
\endminipage
\vspace{0.1cm}
\minipage{\textwidth}
\centering
  \includegraphics[width=\linewidth]{fig/Pics_1_50/EF_1_13.png}
\endminipage\hfill
\caption{Eigenfunctions 12 and 13 of level $m=1$}
\label{fig:ex3}
\end{figure}
\begin{figure}[H]
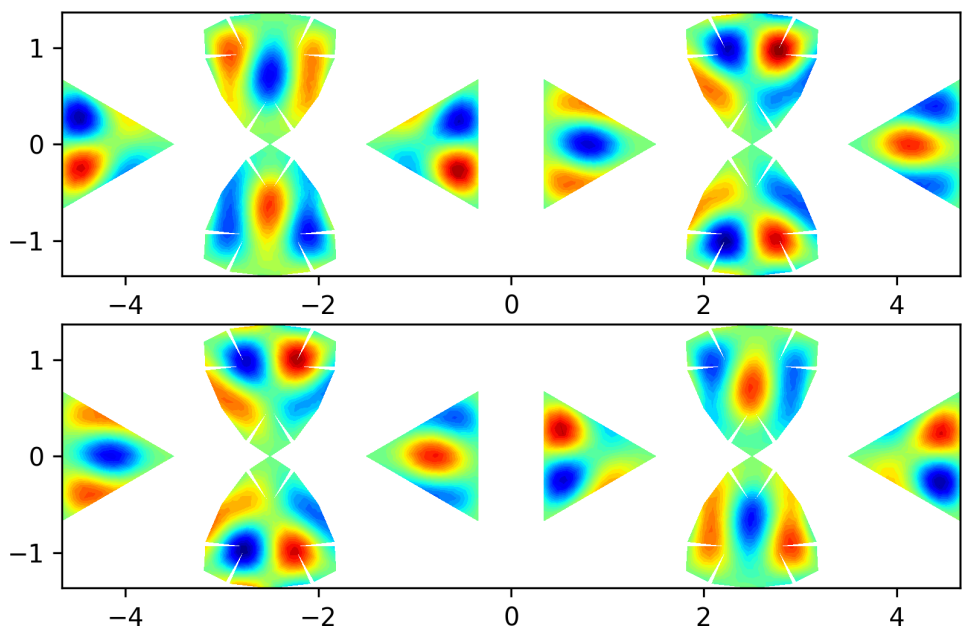

\minipage{\textwidth}
\centering
  \includegraphics[width=\linewidth]{fig/Pics_2_50/EF_2_42.png}
\endminipage
\vspace{0.1cm}
\minipage{\textwidth}
\centering
  \includegraphics[width=\linewidth]{fig/Pics_2_50/EF_2_43.png}
\endminipage\hfill
\caption{Eigenfunctions 42 and 43 of level $m=2$}
\label{fig:ex4}
\end{figure}

\clearpage

There are two types of eigenfunctions from an eigenspace of dimension 3. Eigenfunctions from the first type seem to have rotational symmetry by $\frac{\pi}{2}$ between the tops. One of them has horizontal symmetry, and vertical symmetry with respect to the vertical axes through the tops. The other two have horizontal skew-symmetry through the horizontal axes of one of the tops and vertical symmetry through the other (Fig. \ref{fig:ex5}). For the second type: one has horizontal and vertical skew-symmetry, one has vertical skew-symmetry and rotational skew-symmetry around the tops and the last one has vertical symmetry and rotational skew-symmetry around the tops (Fig. \ref{fig:ex6}).

More data and eigenfunctions can be found at \cite{W}.
\begin{figure}[H]
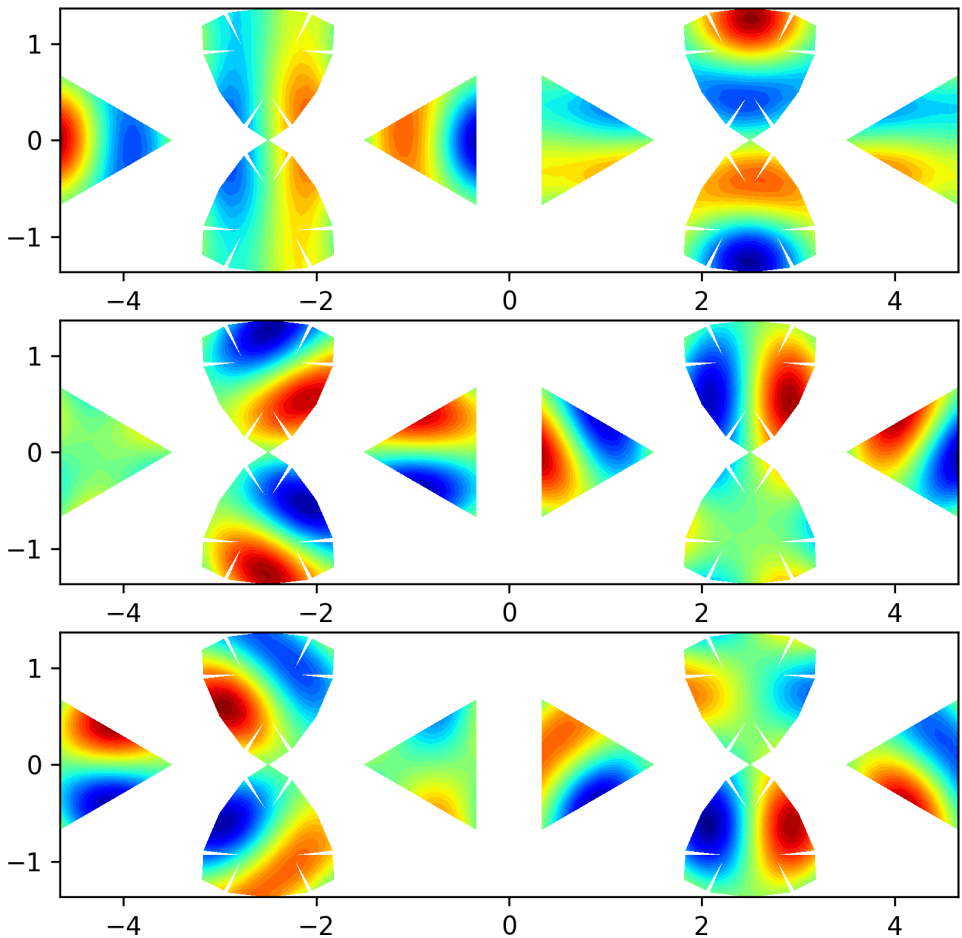

\minipage{\textwidth}
\centering
  \includegraphics[width=\linewidth]{fig/Pics_2_50/EF_2_13.png}
\endminipage
\vspace{0.1cm}
\minipage{\textwidth}
\centering
  \includegraphics[width=\linewidth]{fig/Pics_2_50/EF_2_12.png}
\endminipage
\vspace{0.1cm}
\minipage{\textwidth}
\centering
  \includegraphics[width=\linewidth]{fig/Pics_2_50/EF_2_14.png}
\endminipage\hfill
\caption{Eigenfunctions 12, 13, 14 of level $m=2$}
\label{fig:ex5}
\end{figure}

\clearpage

\begin{figure}[H]
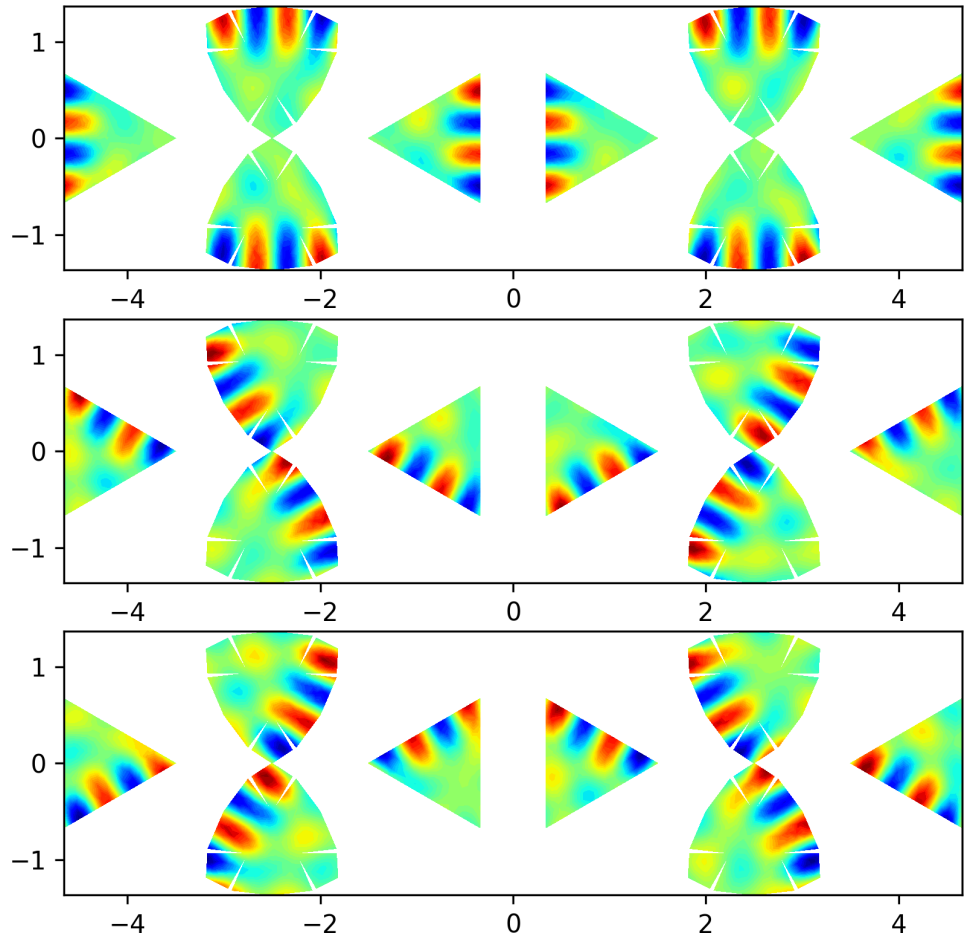

\minipage{\textwidth}
\centering
  \includegraphics[width=\linewidth]{fig/Pics_2_50/EF_2_67.png}
\endminipage
\vspace{0.1cm}
\minipage{\textwidth}
\centering
  \includegraphics[width=\linewidth]{fig/Pics_2_50/EF_2_68.png}
\endminipage
\vspace{0.1cm}
\minipage{\textwidth}
\centering
  \includegraphics[width=\linewidth]{fig/Pics_2_50/EF_2_69.png}
\endminipage\hfill
\caption{Eigenfunctions 67, 68, 69 of level $m=2$}
\label{fig:ex6}
\end{figure}

\clearpage


\bibliographystyle{amsplain}

\end{document}